\documentclass{amsproc}
\usepackage{latexsym,amsmath,amsthm,amssymb}
\RequirePackage{srcltx}

\newtheorem{theorem}{Theorem}
\newtheorem{lemma}[theorem]{Lemma}
\newtheorem{corollary}[theorem]{Corollary}

\theoremstyle{remark}

\newtheorem{remark}[theorem]{Remark}

\theoremstyle{definition}
\newtheorem*{definition}
{Definition}

\newcommand{\dR}{\ensuremath{\mathbb{R}}} 
\newcommand{\R}{\dR}

\begin{document}

\title
{Regularity of the Monge--Amp\`{e}re equation \\ in Besov's spaces}

\author{Alexander~V.~Kolesnikov}
\address{ Higher School of Economics, Moscow,  Russia}
\email{Sascha77@mail.ru}

\vspace{5mm}
\author{Sergey~Yu.~Tikhonov}
\address{
ICREA
 and Centre de Recerca Matem\`{a}tica
\\ Apartat 50 08193 Bellaterra, Barcelona, Spain
}
\email{stikhonov@crm.cat}

\thanks{
This study was carried out within ''The National Research University Higher School of Economics`` Academic Fund Program in 2012-2013, research grant No. 11-01-0175. The first author was supported by RFBR projects 10-01-00518, 11-01-90421-Ukr-f-a, and the program SFB 701 at the University of Bielefeld.
The second author was partially supported by  MTM 2011-27637, 2009 SGR 1303, RFFI 12-01-00169, and NSH-979.2012.1.
}
\date{September 1, 2011}

\vspace{5mm}

\begin{abstract}
Let  $\mu = e^{-V} \ dx$ be a probability measure
and $T = \nabla \Phi$ be the optimal transportation mapping
pushing forward $\mu$ onto  a log-concave compactly supported measure $\nu = e^{-W} \ dx$.
In this paper, we introduce a new approach to the regularity problem for  the corresponding  Monge--Amp{\`e}re
equation $e^{-V} = \det D^2 \Phi \cdot e^{-W(\nabla \Phi)}$ in the Besov spaces $W^{\gamma,1}_{loc}$.
We prove that $D^2 \Phi \in W^{\gamma,1}_{loc}$  provided $e^{-V}$ belongs to a proper Besov class and $W$ is convex.
In particular, $D^2 \Phi \in L^p_{loc}$ for some $p>1$.
Our proof  does not rely on the previously known regularity results.
\end{abstract}
\maketitle

\vspace{5mm}


\section{Introduction}

We consider probability measures
$\mu=e^{-V} \ dx$, $\nu = e^{-W}  \ dx$
on $\R^d$ and the optimal transportation mapping
$T $
pushing forward $\mu$ onto $\nu$
and minimizing the Monge--Kantorovich functional
$$
\int \| x - T(x) \|^2 \ d\mu.
$$
It is known (see, e.g., \cite{Vill1}, \cite{BoKo2012}) that $T$ has the form $T = \nabla \Phi$, where $\Phi$ is a convex function.
If $\Phi$ is smooth, it satisfies the following  change of variables formula
\begin{equation}
\label{VW}
e^{-V} =
e^{-W(\nabla \Phi)} \det D^2 \Phi.
\end{equation}
This relation can be considered as a non-linear second order  PDE with unknown $\Phi$, the so-called Monge--Amp{\`e}re equation.

The regularity problem for the Monge--Amp\`{e}re equation has a rather long history.
The pioneering results have been obtained by  Alexandrov, Bakelman, Pogorelov, Calabi, Yau.
The classical theory can be found in \cite{Pog}, \cite{Bak}, and \cite{GT}.
See also an interesting survey \cite{Kryl} on nonlinear PDE's.

Despite the long history, the sharpest H{\"o}lder regularity results of classical type have been obtained only in the 90'th by
 L.~Caffarelli \cite{Ca1} (see also \cite{CC, Guit, Vill1}).
In particular, Caffarelli proved that $D^2 \Phi$ is H{\"o}lder if $V$ and $W$ are H{\"o}lder on bounded sets $A$ and $B$, where $A$ and $B$ are supports of $\mu$ and $\nu$ respectively.
In addition,  $B$ is supposed to be  convex.
It seems that the latter assumption cannot be dropped as demonstrated by famous counterexamples.
A nice exposition with new simplified proofs and historical overview can be found in \cite{TW}.

Another result from \cite{Ca1} establishes sufficient conditions for $\Phi$ to belong to the second Sobolev class
$W^{2,p}_{loc}$ with $p>1$. More precisely, Caffarelli considered solution of the Monge--Amp{\'e}re equation
$$\det D^2 \Phi =f$$ on a convex set $\Omega$ with $\Phi|_{\partial \Omega}=0$.
Assume that $\Omega$ is normalized:
$B_1 \subset \Omega \subset B_d$ (an arbitrary convex set $\Omega$ can be normalized using an affine transformation).
It is  shown in \cite{Ca1} that for every $p>0$ there exists $\varepsilon(p)>0$ such that if
$|f-1|<\varepsilon(p)$ then $\|\Phi\|_{W^{2,p}(B_{1/2})} \le C(\varepsilon)$.
Wang \cite{Wang} proved that for a fixed $\varepsilon$ in $|f-1|<\varepsilon$
 the value of $p$ in the inclusion $\Phi \in W^{2,p}_{loc}$ cannot be chosen arbitrary large.

This Sobolev regularity result has been extended and generalized in different ways in the recent papers  \cite{Savin}, \cite{DePhil-Fig}, \cite{DePhil-Fig-Sav}, and \cite{Schmidt}. See also \cite{Huang} for some results on the mean oscillation of $D^2 \Phi$.
It was shown in \cite{DePhil-Fig-Sav} that every $\Phi$ satisfying $\det D^2 \Phi =f$ on a normalized convex set $\Omega$ with $\Phi|_{\partial \Omega}=0$ belongs to $W^{2,1+\varepsilon}(\Omega')$, where
$\Omega' = \Big\{x: \Phi(x) \le - \frac{1}{2} \|\Phi\|_{L^{\infty}}\Big\}$ provided $0 < \lambda <  f < \Lambda$.

The main purpose of this paper is to develop an alternative   approach to the regularity problem of the Monge--Amp{\`e}re equation.
We prove that $\Phi$ belongs to a Besov's space under the assumption that
$e^{-V}$ is Besov  and $W$ is convex.   We give a short  proof which does not use previously known regularity results. Our estimates rely on a generalization of the so-called above-tangent formalism which
has been widely used in the applications of the optimal transport theory in probability and PDE's (see  \cite{AGS}, \cite{BoKo2012}, \cite{Vill1}, and \cite{Vill2}). We also apply a result of McCann on the change of variables formula
 and some classical results on equivalence of functional norms.

Some estimates of the type considered in this paper have been previously obtained in \cite{Kol2010} in the case of the Sobolev spaces.
Applications to the infinite-dimensional analysis and convex geometry can be found in \cite{BoKo2011} and \cite{Kol2011-12} respectively.

Hereafter $B_r$ denotes the ball of radius $r$ centered at $0$. We use notation $D^2 \Phi$ for the Hessian matrix of $\Phi$ and $\| \cdot \|$ for the standard operator norm.  We will assume that the measures $\mu$ and $\nu$ satisfy the following assumptions:

\begin{description}
  \item[{\bf Assumption A}] The potential $V: \R^d \to \R$ has the representation
$$V = V_0 + V_1,$$ where $|V_0|$ is globally bounded,  $V_1 $ admits local Sobolev  derivatives,  and  $|\nabla V_1| \in L^1(\mu)$.

  \item[{\bf Assumption B}]
      The support of $\nu$ is a compact convex set $B \subset B_R$.
\end{description}

\begin{definition} {\bf Besov's space (Fractional Sobolev's space).}
 The space $W^{s,p}(Q)$, where $Q$ is a cube in $\R^d$, consists of functions with the finite norm
$$
\|u\|_{W^{s,p}(Q)} = \|u\|_{L^p(Q)} + \Bigl( \int_Q \int_Q \frac{|u(x)-u(y)|^p}{|x-y|^{d+sp}} \, dx\,dy\Bigr)^{\frac{1}{p}}.
$$
 The space $W^{s,p}_0(\R^{d})$ is the completion of $C^{\infty}_0(\R^d)$ with the norm
$$
\|u\|_{W^{s,p}_0(\R^d)} =  \Bigl( \int_{\R^d} \int_{\R^d} \frac{|u(x)-u(y)|^p}{|x-y|^{d+sp}} \, dx\,dy\Bigr)^{\frac{1}{p}}.
$$
\end{definition}

\begin{definition} {\bf Log-concave measure.}
A probability measure $\nu$ is called log-concave if it satisfies the following inequality for all compact sets $A, B$:
$$
\nu( \alpha A + (1-\alpha) B) \ge \nu^{\alpha}(A) \nu(B)^{1-\alpha}
$$
and any $0 \le \alpha \le 1$. If  $\nu$ has a density $\nu = e^{-W} \ dx$, then $W$  must be convex
(we assume that  $W = + \infty$ outside of $\mbox{supp}(\nu)$). This is a classical result of C. Borell \cite{Bor}.
\end{definition}

\begin{remark}
We note that the second derivative of the convex function $\Phi$ can be understood in different ways.
In the generalized (weak) sense this is a measure with absolutely continuous part $D^2_a \Phi \ dx$ and singular part $D^2_s \Phi$.
Throughout the paper we use the following agreement: the statement ``$D^2 \Phi$ belongs to a certain Besov or Sobolev class'' means that the measure $D^2 \Phi$
has no singular component and the corresponding Sobolev derivative $D^2 \Phi = D^2_a \Phi$ belongs to this class.
\end{remark}

\begin{theorem}
\label{main}
Let Assumptions {\bf A} and {\bf B} be fulfilled and, moreover,
\begin{itemize}
\item[(1)]
there exist $p \ge 1$ and $0<\gamma <1$ such that  for every $r>0$
\begin{align*}
\label{main-ass}
\int_{B_r}   \frac{  \| (\delta_y V)_{+} \|_{L^p(\mu)} }{|y|^{d+ \frac{1}{p}  + \gamma} }  \   dy < \infty,
\end{align*}
where $\delta_y V = V(x+y) + V(x-y) - 2V(x)$ and $(\delta_y V )_{+}$ is the non-negative part of $\delta_y V$;
\item[(2)]
$ e^{-V}$ is locally bounded from below;

\item[(3)]
 $\nu$ is a compactly supported log-concave measure.
 \end{itemize}
Then
$$
\| D^2 \Phi \|_{W^{\varepsilon, 1}(Q)} < \infty
$$
for  every cube $Q \subset \R^d$ and $0<\varepsilon < \frac{\gamma}{2}$.
\end{theorem}
Taking $p=+\infty$ we obtain the following result.

\begin{corollary}
Let Assumptions {\bf A}, {\bf B}  and conditions (2)-(3)  be fulfilled.
 Assume, in addition, that $ (\delta_y V)_{+} \le \omega(|y|)$ for some $\omega : \R^+ \to \R^+$  and
$$
\int_{0}^{r}  \frac{  \omega(s)}{s^{1+ \gamma} }  \   ds < \infty
$$
for some $\gamma>0$ and every $r>0$.
Then $
\| D^2 \Phi \|_{W^{\varepsilon, 1}(Q)} < \infty
$
with $0<\varepsilon < \frac{\gamma}{2}$.
\end{corollary}

In particular, applying the fractional Sobolev embedding theorem (see, e.g., \cite[Ch. V]{Adams}) we get
\begin{corollary} \label{lp}
Under assumptions of Theorem \ref{main}, for every $\varepsilon>0$,
$$\|D^2 \Phi\| \in L^{\frac{d}{d-\frac{\gamma}{2} + \varepsilon}}_{loc}.$$
\end{corollary}

\begin{remark}
Note that in corollary \ref{lp} we do not assume that $V$ is bounded from below.
\end{remark}

\vspace{5mm}
\section{Auxiliary results}

Below we will use  the following function space (see, e.g., \cite{Stein}).
\\Let  $\Lambda^{p,q}_{\alpha}$, $0<\alpha <2$, be the space of functions with the finite norm
$$
\|f\|_{\Lambda^{p,q}_{\alpha}} = \|f\|_p+ \Bigl[ \int_{\R^d} \frac{\big(\|f(x+t) + f(x-t) - 2f(x) \|_p\big)^q }{|t|^{d+ \alpha q}} \ dt \,\Bigr]^{\frac{1}{q}},
$$
where $\| \cdot \|_p$ is the $L^p$-norm with respect to the Lebesgue measure. We will apply the following equivalence result from
\cite[Ch.5, Sec. 5, Propos. 8']{Stein}.
\begin{lemma}
\label{stein-th}
For $\alpha>1$ the norm  $\|f\|_{\Lambda^{p,q}_{\alpha}}$ is equivalent to
$$
\|f\|_p + \Bigl[ \int_{\R^d} \frac{(\|\nabla f(x+t) - \nabla f(x) \|_p)^q }{|t|^{d+ (\alpha-1)  q}} \ dt \Bigr]^{\frac{1}{q}}.
$$
In particular, if $\alpha>1$ and $\|f\|_{\Lambda^{p,q}_{\alpha}} < \infty$, then $f$ admits the Sobolev derivatives.
\end{lemma}

Let us recall that every convex function $\Phi$ on $\R^d$ admits the following two types of second derivatives.
\begin{definition}
We say that the measure $\mu_{ev}$ is the  distributional derivative of a  convex function $\Phi$  along unit vectors $e, v$
if the following  integration by parts formula holds for any test function $\xi$:
$$
\int \xi \ d \mu_{ev}  = - \int \partial_e \xi \ \partial_{v} \Phi \ dx.
$$
We set
$$
 \partial_{ev} \Phi :=  \mu_{ev}.
$$
\end{definition}

\begin{definition} The { absolutely continuous part}  $(\partial_{ev} \Phi )_a$  of $\partial_{ev} \Phi $ is called the second Alexandrov derivative of $\Phi$ along $e,v$.
Clearly,
$$
\partial_{ee} \Phi \ge (\partial_{ee} \Phi)_a \ge 0
$$
in the sense of measures. 
\end{definition}
Let us  denote by $D^2_a \Phi$ the matrix consisting on these absolutely continuous parts.
We will apply the following result of R.~McCann from \cite{McCann}.
\begin{theorem}
For $\mu$-almost all $x$ the following change of variables formula  holds
$$
e^{-V(x)} = \det D^2_a \Phi(x) \cdot e^{-W(\nabla \Phi(x))}.
$$
\end{theorem}

We will need the following lemma from \cite{Kol2010}.
In fact, this is a generalization of the  well-known  above-tangent lemma which has numerous applications in probability
and gradient flows of measures with respect to the Kantorovich metric (see \cite{Vill1}, \cite{Vill2}, \cite{AGS}, \cite{BoKo2012}).
The proof follows directly  from the change of variables and  integration by parts. 

\begin{lemma}
\label{basiclemma}
Assume that $W$ is twice continuously differentiable and $$D^2 W \ge K \cdot \mbox{\rm{Id}}$$ for some $K \in \R$, $p \ge 0$. Then
\begin{align*}
\int  & \delta_y V (\delta_{y} \Phi)^p d\mu
\\&
\ge
\frac{K}{2} \int |\nabla \Phi(x+y)- \nabla \Phi(x)|^2 (\delta_{y} \Phi)^p \ d\mu  +
\frac{K}{2} \int |\nabla \Phi(x-y)- \nabla \Phi(x)|^2 (\delta_{y} \Phi)^p \ d\mu \\&
+ p \int \Big\langle \nabla \delta_{y} \Phi , (D^2_a \Phi)^{-1}    \nabla \delta_{y} \Phi  \Big\rangle (\delta_{y} \Phi)^{p-1}  \ d\mu,
\end{align*}
where $\delta_y V = V(x+y) + V(x-y) - 2V(x)$.
\end{lemma}

\begin{corollary}
It follows easily from Lemma \ref{basiclemma} that inequality
$$
\int   \delta_y V (\delta_{y} \Phi)^p d\mu
\ge
p \int \Big\langle \nabla \delta_{y} \Phi , (D^2_a \Phi)^{-1}    \nabla \delta_{y} \Phi  \Big\rangle (\delta_{y} \Phi)^{p-1}  \ d\mu
$$
holds for any log-concave measure   $\nu$. In particular, this holds for the restriction of Lebesgue measure
$\frac{1}{\lambda(A)} \lambda|_{A}$ on a convex subset $A$. In this case $W$ is a constant on $A$ and $W(x)=+\infty$ if $x \notin A$.
\end{corollary}

\begin{remark}
Let us assume that $V$ is twice differentiable and $y = te$ for some unit vector $e$.
Dividing by $t^{2p+2}$ and passing to the limit we obtain
\begin{eqnarray}
\label{lp-sd+}
\int V_{ee}  \Phi_{ee}^p  \ d \mu
&\ge& K
\int \| D^2 \Phi \cdot e\|^2   \Phi_{ee}^p  \ d \mu
\\
\nonumber
&+& p \int \langle (D^2 \Phi)^{-1}  \nabla \Phi_{ee}, \nabla \Phi_{ee} \rangle  \Phi_{ee}^{p-1} \ d \mu.
\end{eqnarray}

Now it is easy to get some of the results of \cite{Kol2010} from  (\ref{lp-sd+}). In particular, applying integration by parts for the left-hand side and H{\"o}lder inequalities, one can easily obtain that
$$
K \| \Phi^{2}_{ee}\|_{L^p(\mu)}  \le \frac{p+1}{2} \|  V^2_e\|_{L^p(\mu)}.
$$

It is worth mentioning that (\ref{lp-sd+}) gives, in fact, an a priori estimate for derivatives of $\Phi$ up to the {\it third} order
(due to the term $p \int \langle (D^2 \Phi)^{-1}  \nabla \Phi_{ee}, \nabla \Phi_{ee} \rangle  \Phi_{ee}^{p-1} \ d \mu$).
\end{remark}

We denote by $\Delta_a \Phi$ the absolutely continuous part of the distributional Laplacian of $\Phi$.

\begin{lemma}
\label{phil1}
Under Assumptions {\bf A} and {\bf B}, there exists $C$ such that
$$\int   \Delta_a \Phi  (x+y) \ d \mu \le \sum_{i=1}^d \int e^{-V(x)} \ d \bigl[ \partial_{x_i x_i} \Phi  (x+y) \bigr] \le C$$ uniformly in $y \in \R^d$.
\end{lemma}
\begin{proof}
The inequality $\int   \Delta_a \Phi  (x+y) \ d \mu \le \sum_{i=1}^d \int e^{-V(x)} \ d \bigl[ \partial_{x_i x_i} \Phi (x+y)\bigr]$
is clear in view of the fact that the singular part of $\partial_{x_i x_i} \Phi$ is  nonnegative.
Moreover, we get
\begin{align*}
& \sum_{i=1}^d \int e^{-V(x)} \ d \bigl[\partial_{x_i x_i}  \Phi (x+y) \bigr]
 \le
 c_1 \sum_{i=1}^d \int e^{-V_1(x)} \ d \bigl[ \partial_{x_i x_i} \Phi  (x+y) \bigr]
\\& = c_1 \int \langle \nabla \Phi(x+y), \nabla V_1(x) \rangle e^{-V_1(x)} \ dx
\le c_1 R \int |\nabla V_1(x)| e^{-V_1(x)} \ dx
\\&
\le c_2 \int | \nabla V_1(x) | \ d \mu.
\end{align*}
\end{proof}

Finally, we will use the fractional Sobolev embedding theorem (see \cite[Ch. V]{Adams}).
We formulate it in the following form given in \cite{MSh} (see also \cite{BBM}, \cite{KolLer}).

\begin{theorem}
Let  $p>1$, $0 < s < 1$ and $sp<d$. Then for every $u \in {W^{s,p}_0(\R^d)}$ one has
$$
\|u\|^p_{L^q(\R^d)} \le c(d,p) \frac{s(1-s)}{(d-sp)^{p-1}}\|u\|^p_{W^{s,p}_0(\R^d)},
$$
where $q = dp/(d-sp)$.
\end{theorem}

\vspace{5mm}
\section{
Proof of Theorem \ref{main}}
Let us apply Lemma \ref{basiclemma} with $p=1$. We have
$$
\int_{\R^d}   \delta_y V  \cdot \delta_{y}  \Phi \ d\mu
\ge
 \int_{\R^d} \Big\langle \nabla \delta_{y} \Phi , (D^2_a \Phi)^{-1}    \nabla \delta_{y} \Phi  \Big\rangle   \ d\mu.
$$
Taking into account Lemma \ref{phil1}, we obtain $\int_{\R^d} \| D^2_a \Phi\| \ d\mu < \infty.$
Then Cauchy inequality yields
$$
\int_{\R^d} \| D^2_a \Phi\| \ d\mu \cdot \int _{\R^d}  \delta_y V  \cdot  \delta_{y}  \Phi \ d\mu
\ge
\Bigl( \int_{\R^d}  | \nabla \delta_{y}  \Phi  | d\mu \Bigr)^2 .
$$
By the H{\"o}lder inequality, for every $p, q \ge 1$, $\frac{1}{p} + \frac{1}{q}=1$, we have
$$
\int _{\R^d}  \delta_y V  \cdot  \delta_{y}  \Phi \ d\mu
\le \int _{\R^d}  (\delta_y V)_{+}  \cdot  \delta_{y}  \Phi \ d\mu
\le \| (\delta_y V)_{+}   \|_{L^p(\mu)} \cdot  \| \delta_{y}  \Phi  \|_{L^q(\mu)}
$$
Let us now estimate $ \| \delta_{y}  \Phi  \|_{L^q(\mu)}$.
Note that  $|\nabla \Phi| \le R$, and hence $ \delta_{y}  \Phi \le 2  R |y|$.

Let us also mention that $t \to \Phi(x+ ty)$ is a one-dimensional convex function for a fixed $x$. Therefore,
$m_y=\partial^2_{tt} \Phi(x + ty)$  is a nonnegative measure on $\R^1$.
 Moreover,
\begin{align*}
\delta_{y}  \Phi  &= \int_0^1 \langle \nabla \Phi(x+ sy) - \nabla \Phi(x-sy) , y \rangle \ ds
\\&= \int_0^1 \int_{-s}^s  d m_y    \ ds.
\end{align*}

Therefore, we have
\begin{align*}
\| \delta_{y}  \Phi  \|^q_{L^q(\mu)} &
= \int ( \delta_{y}  \Phi  )^{q} \ d\mu
\le (2R |y|)^{q-1} \int \delta_{y}  \Phi  \ d\mu
\\&
= (2R |y|)^{q-1} \int_0^1 \int_{-s}^s   \int e^{-V(x)}  d \bigl[ \partial^2_{tt} \Phi(x + ty)   \bigr] dt \ ds.
\end{align*}
It follows from Lemma \ref{phil1} that
$$
\| \delta_{y}  \Phi  \|^q_{L^q(\mu)} \le C |y|^{1+q}.
$$
Hence,
\begin{align*}
\Bigl( \int_{\R^d}  | \nabla \delta_{y}  \Phi  | d\mu \Bigr)^2 \le \int _{\R^d}  \delta_y V  \cdot  \delta_{y}  \Phi \ d\mu  \le
C  \| (\delta_y V)_{+}   \|_{L^p(\mu)} |y|^{1+\frac{1}{q}}.
\end{align*}
Now we divide this inequality by $|y|^{d+ 2+ \gamma}$  and integrate it over  a bounded subset  $Q \subset \R^d$.
By  condition (1), we get
\begin{equation}
\label{frac-sob}
\int_{Q} \frac{1}{|y|^{d+2+\gamma}} \Bigl( \int_{\R^d}  | \nabla   \Phi(x+y) + \nabla \Phi(x-y) - 2\nabla \Phi(x)   | \  d\mu \Bigr)^2  \ dy < \infty.
\end{equation}

Let us take a smooth compactly supported  function  $\xi \ge 0$. We will show that
\begin{equation}
\label{frac-sob2}
\int_{\R^d}  \frac{1}{|y|^{d+2+\gamma}} \Bigl( \int_{\R^d}  |    \xi(x+y) \nabla \Phi(x+y) + \xi(x-y)\nabla \Phi(x-y) - 2 \xi(x) \nabla \Phi(x)   | \  d x \Bigr)^2  \ dy < \infty.
\end{equation}
Let us split this integral in two parts: $$\int_{\R^d} \cdots \ dx= \int_{B_1} \cdots \ dx + \int_{B^c_1} \cdots \ dx
= I_1 + I_2.$$
To estimate the second part, we note that
$$
\int_{\R^d}  |    \xi(x+y) \nabla \Phi(x+y) + \xi(x-y)\nabla \Phi(x-y) - 2 \xi(x) \nabla \Phi(x)    |  \  d x \le 4R \int_{\R^d} |\xi(x)| \ dx.
$$
Thus, $ I_2 < \infty$.

Let us estimate $I_1$. It follows from estimate (\ref{frac-sob}) and condition (3) of the theorem that
$$
\int_{B_1}   \frac{1}{|y|^{d+2+\gamma}} \Bigl( \int_{\R^d}   |   \nabla \Phi(x+y) + \nabla \Phi(x-y) - 2 \nabla \Phi(x)     |  \xi(x)  \  d x \Bigr)^2  \ dy < \infty.
$$
Therefore, it is enough to show that
\begin{multline*}
I_3=\int_{B_1} \frac{1}{|y|^{d+2+\gamma}} \Bigl( \int_{\R^d} |   \nabla \Phi(x+y) \bigl( \xi(x+y) - \xi(x) \bigr) \\+ \nabla \Phi(x-y)  \bigl( \xi(x-y) - \xi(x) \bigr)  |  \  d x \Bigr)^2  \ dy < \infty.
\end{multline*}
Since $|y| \le 1$ and $\xi$ is compactly supported, there exists $R_0>0$ such that
\begin{multline*}
I_3 \le \int_{B_{R_0}} \frac{1}{|y|^{d+2+\gamma}}  \Bigl( \int _{B_{R_0}} \big|   \nabla \Phi(x+y) \bigl( \xi(x+y) - \xi(x) \bigr) \\+ \nabla \Phi(x-y)  \bigl( \xi(x-y) - \xi(x) \bigr)  \big|  \  d x \Bigr)^2  \ dy < \infty.
\end{multline*}
Using smoothness conditions on $\xi$, we obtain
\begin{align*}
\Big| \nabla \Phi(x+y) & \bigl( \xi(x+y) - \xi(x) \bigr)  + \nabla \Phi(x-y)  \bigl( \xi(x-y) - \xi(x) \bigr) \Big|
\\&
=
 \Big|\bigl( \nabla \Phi(x+y) - \nabla \Phi(x-y) \bigr) \bigl( \xi(x+y) - \xi(x) \bigr) \\&\qquad + \nabla \Phi(x-y)  \bigl(  \xi(x+y) + \xi(x-y) - 2\xi(x) \bigr) \Big|
\\&
\le C \Bigl( R_0 |y|^2  + |y| | \nabla \Phi(x+y) - \nabla \Phi(x-y) | \Bigr ).
\end{align*}
Thus, it is sufficient to show that
\begin{equation}\label{delta}
 \int_{B_R} \frac{1}{|y|^{d+\gamma}} \Bigl( \int _{B_R} |   \nabla \Phi(x+y) - \nabla \Phi(x-y) | \  d x \Bigr)^2  \ dy < \infty.
\end{equation}
To prove (\ref{delta}), we use the representation
$$
\int _{B_R} |   \nabla \Phi(x+y) - \nabla \Phi(x-y) | \  d x = \int_{B_R} \Bigr| \int_{-1}^{1} \sum_{i=1}^d  d \bigl[ \partial_{s} \Phi_{e_i}(x + sy)\Bigr] (s) \cdot y_i   \bigl|  \ dx.
$$
The latter is bounded by $C |y| \int_{B_{2R}} \Delta \Phi.$
This immediately implies that $I_3< \infty$ and therefore (\ref{frac-sob2}) is proved.

This means that $|\xi \cdot \nabla \Phi|_{\Lambda^{1,2}_{1+ \gamma/2}} < \infty $ for every smooth compactly supported $\xi$.
Using smoothness conditions on $\xi$ and boundedness of $\nabla \Phi$, we get from  Lemma \ref{stein-th} that $D^2  \Phi$ has no singular parts and
\begin{equation}
\label{besov-est}
\int_{B_r} \frac{1}{|y|^{d + \gamma}} \Bigl( \int_{B_r}  | D^2   \Phi(x+y) - D^2 \Phi(x)   | \  d x \Bigr)^2  \ dy < \infty
\end{equation}
for any $B_r$.

Finally, applying the Cauchy--Schwarz inequality, we obtain for every $\delta>0$
\begin{align*}
 \Bigl( \int_{B_r}  & \frac{1}{|y|^{d + \gamma/2 - \delta/2}}  \int_{B_r}  | D^2   \Phi(x+y) - D^2 \Phi(x)   | \  d x   \ dy \Bigr)^2
\\&
\le
\int_{B_r} \frac{1}{|y|^{d + \gamma}} \Bigl( \int_{B_r}  | D^2   \Phi(x+y) - D^2 \Phi(x)   | \  d x \Bigr)^2  \ dy \cdot  \int_{B_r} \frac{|y|^{\delta}}{|y|^{d}}  \ dy <\infty, \ \ \forall B_r
\end{align*}
Changing variables implies
$$
\int_{Q}  \int_{Q} \frac{  | D^2   \Phi(z) - D^2 \Phi(x)   | }{|z-x|^{d + \varepsilon}}  \  d x   \ dz < \infty
$$
for every $0<\varepsilon < \frac{\gamma}{2}$ and bounded $Q$. The proof is now complete.
\hfill $\square$

\vspace{5mm}

\section{Remarks on improved integrability}

By applying Theorem \ref{main} we get  a better (local) integrability of $\|D^2 \Phi\|$ ($L^{\frac{d}{d-\frac{\gamma}{2} + \varepsilon}}$ instead of $L^1$); see Corollary \ref{lp}. This can be used to improve the estimates obtained in Theorem \ref{main}.

We assume  for simplicity that we transport measures  with periodical densities
by a periodical optimal mapping
$$
T(x) = x + \nabla \varphi(x),
$$
where $\varphi$ is periodical. Equivalently, one can consider optimal transportation of
probability measures on the flat torus $\mathbb{T}$.
Then one can repeat the above arguments and obtain the same estimates which become
{\it global}.

In general, it can be shown that the assumption $\|D^2 \varphi\| \in L^r(\mathbb{T})$ implies
 $$
\int_\mathbb{T} \int_\mathbb{T}
\frac{|D^2 \varphi(x+y)-D^2 \varphi(x)|^{\frac{2r}{r+1}} }
{|y|^{d+ \frac{r}{r+1}(\gamma + \frac{r-1}{q}) -\varepsilon}} \ dx \ dy < \infty
$$
for every $\varepsilon$ and $\|D^2 \varphi \| \in L^{r'}(\mathbb{T})$
with any $r'$ satisfying
$$
 r' < \frac{\frac{2 d r}{r+1}}{d - \frac{r}{r+1} (\gamma+\frac{r-1}{q})}.
$$
Starting with $r_0=1$ and iterating this process one can obtain
a sequence $r_n$ such that $\|D^2 \varphi\| \in L^{r_n-\varepsilon}$ for every $r_n$ and $\varepsilon>0$
$$
r_0=1, \ \ r_{n+1} = \frac{\frac{2 d r_n}{r_n+1}}{d - \frac{r_n}{r_n+1} (\gamma+\frac{r_n-1}{q})}.
$$
One has $\|D^2 \varphi\| \in L^{r-\varepsilon}$, where $r = \lim_n r_n$
solves the equation
$$
x^2 + x(q(\gamma -d) -1)+qd=0
$$
with $r>1$.

\end{document}